\newcommand{\arxiv}{}

\ifdefined\arxiv
\documentclass[a4paper,UKenglish,cleveref,thm-restate]{lipics-v2021}
\nolinenumbers
\hideLIPIcs
\else
\documentclass[a4paper,UKenglish,cleveref,thm-restate,anonymous]{lipics-v2021}
\fi

\title{On Phases of Unique Sink Orientations}

\author{Michaela Borzechowski}{Department of Mathematics and Computer Science, Freie Universität Berlin, Germany}{michaela.borzechowski@fu-berlin.de}{}{German Research Foundation DFG within the Research Training Group GRK~2434 \emph{Facets of Complexity}}

\author{Simon Weber}{Department of Computer Science, ETH Zürich, Switzerland}{simon.weber@inf.ethz.ch}{https://orcid.org/0000-0003-1901-3621}{Swiss National Science Foundation under project no. 204320.}

\authorrunning{M. Borzechowski and S. Weber} 

\Copyright{Michaela Borzechowski and Simon Weber} 



\category{} 

\relatedversion{} 

\EventEditors{John Q. Open and Joan R. Access}
\EventNoEds{2}
\EventLongTitle{42nd Conference on Very Important Topics (CVIT 2016)}
\EventShortTitle{CVIT 2016}
\EventAcronym{CVIT}
\EventYear{2016}
\EventDate{December 24--27, 2016}
\EventLocation{Little Whinging, United Kingdom}
\EventLogo{}
\SeriesVolume{42}
\ArticleNo{23}

\usepackage{enumitem}
\usepackage{color}
\definecolor{orange}{rgb}{0.898, 0.621, 0.0}
\definecolor{skyblue}{rgb}{0.336, 0.703, 0.910}
\definecolor{green}{rgb}{0, 0.617, 0.449}
\definecolor{yellow}{rgb}{0.937, 0.890, 0.258}
\definecolor{blue}{rgb}{0, 0.445, 0.695}
\definecolor{red}{rgb}{0.832, 0.367, 0}
\definecolor{purple}{rgb}{0.797, 0.473, 0.652}

\usepackage{xstring}
\usepackage{ifthen}
\usepackage{tikz}
\usepackage{xspace}

\newcommand{\f}[1]{\relax\ifmmode#1\else{$#1$}\fi}

\newcommand{\szabo}{Szab{\'o}}
\newcommand{\SWC}{\szabo-Welzl condition\xspace}
\newcommand{\xor}{\oplus}

\newcommand{\PSPACE}{\ensuremath{\mathsf{PSPACE}}\xspace}
\newcommand{\NPSPACE}{\ensuremath{\mathsf{NPSPACE}}\xspace}

\newcommand{\flip}{\otimes}

\newcommand{\neig}{\ominus}

\newcommand{\Phase}{\f{P}\xspace}
\newcommand{\Matching}{\f{H}\xspace}
\newcommand{\neighborhoodGraph}{\f{N}\xspace}
\newcommand{\orientation}{\f{O}\xspace}
\newcommand{\dimension}{\f{n}\xspace}
\newcommand{\Cube}{\f{Q_\dimension}\xspace}
\newcommand{\face}{\f{f}\xspace}
\newcommand{\facet}{\f{F}\xspace}
\newcommand{\dimOfFace}{\f{dim(f)}\xspace}
\newcommand{\node}{\f{v}\xspace}
\newcommand{\nodeA}{\f{v}\xspace}
\newcommand{\nodeB}{\f{w}\xspace}

\newcommand{\edge}{\f{e}\xspace}
\newcommand{\edgeA}{\f{e}\xspace}
\newcommand{\edgeB}{\f{e'}\xspace}

\renewcommand{\P}{\mathcal{P}}
\newtheorem{fact}[theorem]{Fact}

\begin{document}

\maketitle

\begin{abstract}
    A \emph{unique sink orientation} (USO) is an orientation of the $n$-dimensional hypercube graph such that every non-empty face contains a unique sink. Schurr showed that given any $n$-dimensional USO and any dimension~$i$, the set of edges $E_i$ in that dimension can be decomposed into equivalence classes (so-called \emph{phases}), such that flipping the orientation of a subset $S$ of $E_i$ yields another USO if and only if $S$ is a union of a set of these phases. In this paper we prove various results on the structure of phases. Using these results, we show that all phases can be computed in $O(3^n)$ time, significantly improving upon the previously known $O(4^n)$ trivial algorithm. Furthermore, we show that given a boolean circuit of size $poly(n)$ succinctly encoding an $n$-dimensional (acyclic) USO, it is \PSPACE-complete to determine whether two given edges are in the same phase. The problem is thus equally difficult as determining whether the hypercube orientation encoded by a given circuit is an acyclic USO [Gärtner and Thomas, STACS'15].
\end{abstract}

\ifdefined\arxiv
\else
\clearpage
\fi
\section{Introduction}
A \emph{unique sink orientation} (USO) is an orientation of the $n$-dimensional hypercube graph, such that for each non-empty face of the hypercube, the subgraph induced by the face contains a unique sink, i.e., a unique vertex with no outgoing edges.

USOs were originally proposed by Stickney and Watson~\cite{stickney1978digraph} as a way of modelling the candidate solutions of instances of the \emph{P-matrix Linear Complementarity Problem}~(P\nobreakdash-LCP). After USOs were then forgotten for more than two decades, they were reintroduced and formalized as purely combinatorial objects by \szabo{} and Welzl in 2001~\cite{szabo2001usos}. Since then, USOs have been studied from many angles. Many problems in geometry, game theory, as well as mathematical programming have since been reduced to the problem of finding the global sink of a unique sink orientation~\cite{borzechowski2023degeneracy,fischer2004smallestball,foniok2009pivoting,gaertner2005stochasticgames,gaertner2006lpuso,halman2007stochasticgames,schurr2004phd}. Motivated by this, much research has gone into finding better algorithms for this problem~\cite{gaertner1995subexaop}, however the original algorithms by \szabo{} and Welzl with runtimes exponential in $n$ are still the best-known in the general case.

Much of the existing research on USOs has gone into structural and combinatorial aspects. 
On the $n$-dimensional hypercube, there are $2^{\Theta(2^n\log n)}$ USOs~\cite{matousek2006numberusos}. While this is a large number, it is still dwarfed by the number of all orientations, which is $2^{n2^{n-1}}$. This shows that the USO condition is quite strong, yet still allows for a large variety of orientations. However, this also makes it quite challenging to enumerate USOs or even randomly sample a USO. The naive approach to random sampling --- generating random orientations until finding a USO --- fails since USOs are so sparse among all orientations. A natural approach is thus the one of \emph{reconfiguration}.

In reconfiguration, one defines simple (often ``local'') operations which can be applied to the combinatorial objects under study. 
This defines a so-called \emph{flip graph}, whose vertices represent the combinatorial objects, where neighboring objects can be turned into each other by applying such a simple operation. Enumeration of all objects can then be performed by systematically walking through a spanning tree of the flip graph, e.g., with a technique due to Avis and Fukuda named \emph{reverse-search}~\cite{reversesearch}. If the flip graph is Hamiltonian, we can even find a so-called \emph{Gray code}. Similarly, random sampling among the objects can be performed by simulating a random walk on the flip graph, which is often expressed as a \emph{Markov chain}. Flip graphs have been studied for many types of combinatorial objects, and this study has in many cases contributed to a better understanding of the objects themselves~\cite{permutationLanguages,reconfigurationSurvey}.

On USOs, a natural choice of operation is to flip the orientation of a single edge.
However, it is known that there exist USOs in which no single edge can be flipped without destroying the USO condition~\cite{borzechowski2022construction,schurr2004phd}.
Thus, to guarantee connectedness of the flip graph, we need operations that may flip multiple edges at once.

This problem was intensely studied by Schurr~\cite{schurr2004phd}. In any USO $O$ of the $n$-cube, and any dimension $i\in[n]$, we denote the set of edges of $O$ in dimension $i$ by $E_i$. Schurr proved that we can then decompose $E_i$ into equivalence classes, such that flipping any subset $S\subseteq E_i$ of edges in $O$ yields another USO if and only if $S$ is the union of some of these equivalence classes. Schurr named these equivalence classes within $E_i$ the \emph{$i$-phases} of $O$.

It turns out that phases are very useful for reconfiguration. As the operation we define the following: For any dimension $i\in[n]$, flip the set of edges given by some set of $i$-phases. Schurr showed that with only $2n$ of these operations we can obtain any $n$-dimensional USO from any other. The flip graph is thus connected and has rather low diameter. Furthermore, Schurr showed that the naturally defined Markov chain based on this operation converges to the uniform distribution. However, it is neither known whether the flip graph is Hamiltonian nor how quickly the Markov chain converges.

Since this is the only known connected flip graph for USOs, it is crucial to understand it better. However, to understand this flip graph and the associated Markov chain we must first gain a better understanding of the structure of phases themselves. In this paper, we make
significant progress on that front by presenting several surprising structural properties of phases. We also show some consequences of these properties to algorithmic and complexity-theoretic aspects of the problem of computing phases.

\subsection{Results}
This paper begins with proofs of various structural properties of phases. Specifically, in \Cref{sec:connectedness}, we show that for every phase $P$, the subgraph of the hypercube induced by the endpoints of the edges in $P$ is \emph{connected}. In \Cref{sec:hypervertices}, we prove various results about the relationship between phases and \emph{hypervertices}, i.e., faces in which for every vertex the orientation of the edges leaving the face is the same. In \Cref{sec:SchurrsProposition}, we prove that flipping a \emph{matching} in a USO leads to another USO if and only if the matching is a union of phases. This statement was previously claimed by Schurr~\cite[Proposition 4.9]{schurr2004phd}, however his proof of the ``only if'' direction contained severe mistakes that 
were remarkably difficult to repair, requiring the use of newer results on \emph{pseudo USOs}~\cite{bosshard2017pseudo}.
Finally, in \Cref{sec:dSchurrCube}, we show that to compute the phases of an $n$-dimensional USO by Schurr's method, it is not sufficient to compare only neighboring edges or even only edges of some bounded distance. We construct a family of USOs in which one needs to compare antipodal vertices with each other.

In \Cref{sec:ComputationOfPhases}, we provide an algorithm to compute all phases of a given $n$-dimensional USO using $O(3^n)$ vertex comparisons, improving upon the currently best known method due to Schurr which takes $O(4^n)$ comparisons. This algorithmic improvement is then contrasted by our following main result proven in \Cref{sec:completeness}:
\begin{theorem}
    Given a USO $O$ in succinct circuit representation, and two edges $e,e'$, the problem of deciding whether $e$ and $e'$ are in the same phase is \PSPACE-complete, even if $O$ is guaranteed to be acyclic.
\end{theorem}

\section{Preliminaries}

\subsection{Hypercubes and Orientations}


The \emph{\dimension-dimensional hypercube graph} \Cube (\emph{\dimension-cube}) is the undirected graph on the vertex set \mbox{$V(\Cube) = \{0,1\}^\dimension$}, where two vertices are connected by an edge if they differ in exactly one coordinate. 
On bitstrings we use \f{\xor} for the bit-wise \emph{xor} operation and \f{\wedge} for the bit-wise \emph{and}. 
For simplicity we write a dimension $i$ in the subcube spanned by two vertices $v, w\in V(\Cube)$ as $i \in v \xor w$, even though it would be more formally correct to use $i \in \{j\;|\;j \in [n], (v\xor w)_j = 1 \}$.
For a bitstring $v\in\{0,1\}^\dimension$ and a number $i\in [\dimension]$, we use the simplified notation $v\neig i = v\xor I_i$, where $I_i$ is the $i$-th standard basis vector. Thus, for a vertex \node and a dimension $i\in [\dimension]$, the vertex $\node \neig i$ is the neighbor of \node which differs from \node in coordinate $i$. 
The edge between \node and $\node\neig i$ is called an \emph{$i$-edge}, or an \emph{edge of dimension \(i\)}. We denote the set of $i$-edges by $E_i$.

\begin{definition}
A \emph{face} of \Cube described by a string $\face \in\{0,1,*\}^\dimension$ is the induced subgraph of $\Cube$ on the vertex set $V(\face):= \{\node \in V(\Cube)\;|\;\forall i \in [\dimension] : \node_i=\face_i \text{ or }\face_i=* \}$.  
We write $\dimOfFace$ for the set of dimensions \emph{spanning} the face, i.e., dimensions $i$ for which $\face_i=*$. The \emph{dimension} of $\face$ is $|\dimOfFace|$.
\end{definition}
\begin{definition}
A face of dimension $\dimension-1$ is called a \emph{facet}. 
The facet described by the string $f$ with $f_i=1$ and $f_j=*$ for $j\neq i$ is called the \emph{upper $i$-facet}. Its opposite facet (described by $f_i=0$, $f_j=*$ for $j\neq i$) is called the \emph{lower $i$-facet}.
\end{definition}

An \emph{orientation} $O$ is described by a function $\orientation: V(\Cube) \rightarrow \{0,1\}^\dimension$ such that for all $v\in V(Q_\dimension)$ and all $i\in [\dimension]$, $O(v)_i\neq O(v\neig i)_i$. This function assigns each vertex an orientation of its incident edges, called the \emph{outmap} of the vertex, where $\orientation(\node)_i=1$ denotes that the $i$-edge incident to vertex \node is outgoing from \node, and $\orientation(\node)_i=0$ denotes an incoming edge.

\subsection{Unique Sink Orientations}
\begin{definition}
An orientation $\orientation$ of \Cube is a \emph{unique sink orientation (USO)} if within each face $\face$ of~$\Cube$, there exists exactly one vertex with no outgoing edges. That is, there is a unique sink in each face with respect to the orientation \orientation.
\end{definition}

\szabo{} and Welzl~\cite{szabo2001usos} provide a useful characterization for USOs:
\begin{lemma}[\szabo{}-Welzl Condition \cite{szabo2001usos}]\label[lemma]{lem:szabowelzl}
An orientation \orientation of \Cube is a USO if and only if for all pairs of distinct vertices $\nodeA, \nodeB \in V(\Cube)$, we have
\[(\nodeA \oplus \nodeB) \wedge (\orientation(\nodeA) \oplus \orientation(\nodeB)) \neq 0^\dimension.\]
\end{lemma}

This also implies that the function $O$ must be a bijection, even when the domain is restricted to any face $f$ and the codomain is restricted to the dimensions spanned by $f$.

\subsection{Pseudo USOs}
Bosshard and Gärtner \cite{bosshard2017pseudo} introduced the concept of \emph{pseudo USOs} to capture orientations that are \emph{almost} USOs.
\begin{definition}
A \emph{pseudo unique sink orientation} (PUSO) of \Cube is an orientation \orientation that does not have a unique sink, but every proper face $\face \neq \Cube$ has a unique sink.
\end{definition}
Every orientation that is not a USO must contain a minimal face that does not have a unique sink. This minimal face must then be a PUSO. We use this fact in several of our proofs, in conjunction with the following property of PUSOs from \cite{bosshard2017pseudo}.
\begin{lemma}[{\cite[Corollary 6 and Lemma 8]{bosshard2017pseudo}}]\label[lemma]{lem:PUSO}
Let \orientation be a PUSO. Then, for every pair of antipodal vertices \nodeA and \nodeB it holds that $O(v)=O(w)$.
\end{lemma}

\subsection{Phases}
Let $O$ be an orientation, and let $S$ be a set of edges. We write $O\flip S$ for the orientation $O$ with the orientation of all edges in $S$ \emph{flipped}, i.e., their orientations are reversed. In general, given a fixed USO $O$, it is quite difficult to characterize which sets of edges $S$ lead to $O\flip S$ being a USO again. In fact, this problem is equivalent to characterizing the set of all USOs. However, the task becomes much easier if we require $S$ to consist of only $i$-edges for some dimension $i$, i.e., $S\subseteq E_i$.

Schurr~\cite{schurr2004phd} called the minimal sets $S\subseteq E_i$ such that $O\flip S$ is a USO \emph{phases.} It turns out that phases form a partition of $E_i$. Furthermore Schurr proved that if $S$ is a union of phases, $O\flip S$ is also a USO.

Formally, the phases are the equivalence classes of the equivalence relation on $E_i$ obtained by taking the transitive closure of the relation of \emph{direct-in-phaseness}:

\begin{definition}\label[definition]{def:DIP}
Let \orientation be an \dimension-dimensional USO and $i \in [\dimension]$ a dimension.
Two $i$-edges $\edgeA, \edgeB \in E_i$ are \emph{in direct phase} if there exist $\nodeA \in \edgeA$ and $\nodeB \in \edgeB$ such that
\begin{equation}\label{eqn:dip}
\orientation(\nodeA)_j = \orientation(\nodeB)_j \text{ for all } j\in (v\xor w)\setminus\{i\}.
\end{equation}
\end{definition}

In other words, $e$ and $e'$ are in direct phase if two of their incident vertices have the same outmap within the subcube they span, apart from the orientation of the $i$-edges.
We can thus see that in this case, $e$ and $e'$ must be oriented in the same direction in \orientation. Flipping just one of the two edges leads to an immediate violation of the \SWC by $v$ and $w$, the orientations $O \flip \{e\}$ and $O\flip\{e'\}$ are not USOs.

Further note that $v$ and $w$ must lie on opposing sides of their respective $i$-edges. However, not both pairs of opposing endpoints of \edgeA and \edgeB need to fulfill \Cref{eqn:dip}. See \cref{fig:onlyOneVertexCertificate} for an example of a USO in which two edges are in direct phase but this fact is only certified by one pair of opposing endpoints.

\begin{figure}[htb]
\centering
\ifdefined\arxiv
\includegraphics[]{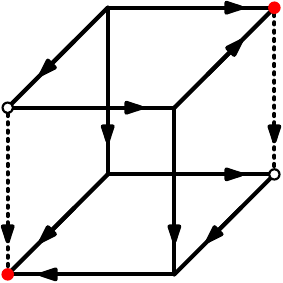}
\else
\includegraphics[scale=0.6]{img/onlyOneVertexCertificate.pdf}
\fi
\caption{The dotted edges are in direct phase. This is certified by the pair of vertices marked in solid red, but not by the vertices circled in black. Note that the phase containing the two dotted edges also contains the front right vertical edge. The back left vertical edge is not in this phase, it is flippable.}
\label{fig:onlyOneVertexCertificate}
\end{figure}

\begin{definition}\label{def:phases}
Let $O$ be a USO. Two edges $e,e'\in E_i$ are \emph{in phase} if there exists a sequence 
\[e=e_0,e_1,\ldots,e_{k-1},e_k=e'\]
such that for all $j\in [k]$, $e_{j-1}$ and $e_j$ are in direct phase.
An \emph{$i$-phase} $P\subseteq E_i$ is a maximal set of edges such that all $e,e'\in P$ are in phase.
We write $\P_i(O)$ for the family of all $i$-phases of $O$.
\end{definition}

On the one hand, since edges that are in direct phase must be flipped together, we can never flip a set $S\subseteq E_i$ that is not a union of $i$-phases. On the other hand, if flipping some set $S\subseteq E_i$ destroys the \SWC for some vertices $v,w$, their incident $i$-edges must have been in direct phase, and exactly one of the two edges is in $S$. Therefore, flipping a union of $i$-phases always preserves the \SWC. We thus get the following observation, which we will strengthen in \Cref{sec:SchurrsProposition}:

\begin{observation}\label{obs:schurrpropositionSingleDim}
Let $O$ be a USO, and $i\in[n]$ a dimension. For any $S\subseteq E_i$, $\orientation \flip S$ is a USO if and only if $S$ is a union of phases, i.e., $S=\bigcup_{P\in \P'}P$ for $\P'\subseteq \P_i(O)$.    
\end{observation}
Note that this also implies that flipping a union of $i$-phases does not change the $i$-phases.
\ifdefined\arxiv
We further want to note the following:
\begin{observation}
\label{obs:EveryPairOfVerticesCertifiesDIPOfAtMost2Edges}
Every pair of vertices can only certify at most one pair of edges to be in direct phase.
In particular, if two vertices $v$ and $w$ certify that their incident $i$-edges are in direct phase, for all other dimension  $j\in (v\xor w)\setminus\{i\}$, the $j$-edges incident to $v$ and $w$ have opposing orientations and thus cannot be in the same phase.
\end{observation}
\fi


A special case is a phase $P=\{e\}$, i.e., a phase that consists of only a single edge. We call such an edge \emph{flippable}. See \Cref{fig:onlyOneVertexCertificate} again for an example of a flippable edge. Schurr provides the following characterization of flippable edges:

\begin{lemma}[{\cite[Lemma 4.13]{schurr2004phd}}]\label{lem:flippable}
    An edge $\{\node ,\node \neig i\}$ in a USO \orientation is \emph{flippable} if and only if
    $\node$ and $\node \neig i$ have the same outmap apart from their connecting $i$-edge, i.e., 
\[\forall j \in [\dimension] \setminus \{i\} : \orientation(\node)_j=\orientation(\node \neig i)_j.\]
\end{lemma}

There exist \dimension-dimensional USO in which every edge is flippable, i.e., USOs that have $\dimension 2^{\dimension-1}$ phases. This happens for example in the \emph{uniform} USO, where for every $i\in [n]$, every $i$-edge is oriented towards the lower $i$-facet. In contrast, it is also possible that all $i$-edges are in phase with each other, and thus there only exists a single $i$-phase. However, this cannot happen in all dimensions simultaneously: By \cite[Lemma 4.14]{schurr2004phd}, for $n>2$, every $n$-dimensional USO has at least $2n$ phases.

\section{Structural Properties}

In this section we show various new insights on the structure of phases.

\subsection{Connectedness of Phases}\label{sec:connectedness}
As a first question, we investigate whether it is possible that a phase consists of multiple ``components'' that are far apart.
We first show that each edge that is not flippable is in direct phase with at least one ``neighboring'' edge.
\begin{lemma}
Let $e=\{v,v\neig i\}$ be an $i$-edge. If $e$ is not flippable, there exists a dimension $j\neq i$, such that $e$ is in direct phase with the \emph{neighboring} edge $\{v\neig j,(v\neig i)\neig j\}$.
\end{lemma}
\begin{proof}
Since $e$ is not flippable, by \Cref{lem:flippable} there exists at least one dimension $j \in [\dimension]\setminus \{i\}$ in which some orientation of \nodeA and $(\nodeA \neig i)$ disagrees, i.e., $\orientation(\nodeA)_j \neq \orientation(\nodeA \neig i)_j$. Thus, in the 2-face spanned by $v$ and $(v\neig i)\neig j$ the $j$-edges are oriented in the opposite direction. These two vertices certify that $e$ and $\{\nodeA \neig j, (\nodeA \neig i)\neig j\}$ are in direct phase.
\end{proof}

However, edges that are in direct phase with each other are not necessarily neighboring, as can be seen for example in \cref{fig:onlyOneVertexCertificate}. 
Nonetheless, we will prove that with respect to this neighboring relation, every phase is \emph{connected}. Let us first define the neighboring relation more formally.

\begin{definition}
\label{def:neighborhoodGraph}
For some $i\in [\dimension]$ of a cube \Cube, we say two $i$-edges are \emph{neighboring} when there exists a 2-face containing both of them. Let $\neighborhoodGraph_i$ be the graph with $V(\neighborhoodGraph_i) = E(\Cube)_i$. There is an edge between two vertices of $\neighborhoodGraph_i$ if the vertices correspond to neighboring $i$-edges in \Cube.
We call $\neighborhoodGraph_i$ the \emph{neighborhood graph} in dimension $i$.
\end{definition}

\begin{theorem}\label{thm:connectedness}
Let \Phase be an $i$-phase of a USO \orientation. Then the subgraph of $N_i$ induced by the edges of \Phase is connected.
\end{theorem}
\begin{proof}
Let $C,\overline{C}$ be any non-trivial partition of the edges of $P$.
Let $\orientation'$ be the orientation in which we flip all edges in $C$ but not the edges in $\overline{C}$, i.e., $\orientation' := \orientation \flip C$.
The orientation $\orientation'$ can obviously not be USO, since the edges of $C$ and $\overline{C}$ are in phase.
This means that there is a minimal face $f$ such that $\orientation'$ is a PUSO on $f$. It is easy to see that $f$ must span dimension $i$.
By \Cref{lem:PUSO}, every antipodal pair of vertices in $f$ has the same outmap within the dimensions spanned by $f$, i.e., any such pair of vertices fails the \SWC. Therefore, within $f$, every pair of antipodal $i$-edges is in direct phase, and every such pair consists of exactly one edge in $C$ and one edge in $\overline{C}$.
We can thus see that $f$ must contain at least one $i$-edge in $C$ that is neighboring an $i$-edge in $\overline{C}$. Since this holds for all non-trivial partitions of $P$, the subgraph of $N_i$ induced by $P$ must be connected.
\end{proof}
Note that if we would instead consider the subgraph of $N_i$ in which we only use edges $\{e_1,e_2\}$ such that $e_1$ and $e_2$ are both neighboring \emph{and} in direct phase, we could have phases with a disconnected induced subgraph. An example of this is \Cref{fig:3DSchurr}, where none of the front two vertical edges is in direct phase with their neighboring vertical edge in the back. We will elaborate more on this in \Cref{sec:dSchurrCube}.

The connectedness of phases will prove to be a useful tool in the next section where we prove \Cref{thm:AllHEdgesAreInPhase=>Hypervertex}, a connection between phases and \emph{hypervertices}.


\subsection{Phases and Hypervertices}\label{sec:hypervertices}

A face \face of a cube \Cube is called a \emph{hypervertex}, if and only if for all vertices $\nodeA, \nodeB \in \face$ and all dimensions $i$ not spanned by $f$, we have $\orientation(\nodeA)_i = \orientation(\nodeB)_i$.
In other words, for each dimension~$i$, all $i$-edges between $f$ and the rest of the cube are oriented the same way.

By \Cref{lem:flippable}, a hypervertex of dimension $1$ is thus a flippable edge. We therefore know that we can change the orientation within a one-dimensional hypervertex arbitrarily without destroying the USO condition. The following lemma due to Schurr and \szabo{} generalizes this to higher-dimensional hypervertices.

\begin{lemma}[{\cite[Corollary 6]{schurr2004quadraticbound}}]
\label{lem:exchangeHypervertices}
Let $O$ be a $n$-dimensional USO, and \face some $k$-dimensional hypervertex. Then, the orientation on the face \face can be replaced by any other $k$-dimensional USO $\orientation'_\face$, such that the resulting orientation~$\orientation'$ defined by
\[\forall v\in V(\Cube),i\in [\dimension]: \orientation'(v)_i:= \begin{cases}
    \orientation(v)_i & \text{if }\face_i\not=* \text{ or } v\not\in \face, \\
    \orientation'_\face(\{v_j \;|\; \face_j = *\})_i & \text{otherwise},
\end{cases}
\]
is also a USO.
\end{lemma}


Hypervertices thus allow for a local change of the orientation. We will now investigate what this implies about the structure of phases.
In particular, we prove the following two statements, 
given an $n$-dimensional USO $O$ and some face \face:

\begin{lemma}\label{lem:hypervertexIFFphases}
    The face $f$ is a hypervertex if and only if for all $i\in\dimOfFace$, $\edge \in E(\face)_i$ and $\edgeB \in E(\Cube \setminus \face)_i$, \edge and \edgeB are not in phase.
\end{lemma}

\begin{lemma}
\label{thm:AllHEdgesAreInPhase=>Hypervertex}
If there exists some $i \in \dimOfFace$ such that $E(f)_i$ is an $i$-phase of $O$, 
then \face is a hypervertex.
\end{lemma}


In other words, \Cref{lem:hypervertexIFFphases}
says that a face is a hypervertex if and only if the phases of all edges of~\face are strictly contained within that face.
\Cref{thm:AllHEdgesAreInPhase=>Hypervertex} gives a slightly different sufficient condition for $f$ being a hypervertex: \face is a hypervertex if \emph{all} edges of \emph{one} dimension of \face are are exactly \emph{one} phase.

\begin{proof}[Proof of \Cref{lem:hypervertexIFFphases}]
We first prove the ``if'' direction. Assume \face is not a hypervertex. 
Then \face has at least one incoming and one outgoing edge in some dimension $i \notin \dimOfFace$.
More specifically there exists a 2-face crossing between $f$ and $Q_n\setminus f$ that has the aforementioned incoming and outgoing edge, i.e., a 2-face with $\{\nodeA, \nodeB\} \in E(\face)_j$,  $\{(\nodeA \neig i), (\nodeB \neig i)\}  \in E(\Cube \setminus \face)_j$ and $\orientation(\nodeA)_i \neq \orientation(\nodeB)_i$.
This however implies that the edge $\{\nodeA, \nodeB\}$ is in direct phase with the edge $\{(\nodeA \neig i), (\nodeB \neig i)\}$.

Next, we prove the ``only if'' direction and assume \face is a hypervertex. Let $i \in \dimOfFace$, $\edge \in E(\face)_i$ and $\edgeB \in E(\Cube \setminus \face)_i$.
By \Cref{lem:exchangeHypervertices} we can change the current orientation $O_f$ of a hypervertex \face to any arbitrary USO of the same dimension. We can thus replace it by the USO $O_f'=O_f\flip E(\face)_i$, i.e., $O_f$ with all edges of dimension $i$ flipped. This flips \edge but not \edgeB, and the result is still USO. Thus, by \Cref{obs:schurrpropositionSingleDim}, \edge and \edgeB are not in phase.
\end{proof}

\ifdefined\arxiv

To prove \Cref{thm:AllHEdgesAreInPhase=>Hypervertex} we need the connectedness of phases (\Cref{thm:connectedness}) from the previous subsection, as well as another ingredient, the \emph{partial swap}: Given a USO \orientation and a dimension $j\in [\dimension]$, the \emph{partial swap} is the operation of swapping the subgraph of the upper $j$-facet induced by the endpoints of all \emph{upwards oriented} $j$-edges with the corresponding subgraph in the lower $j$-facet. By \cite{borzechowski2022construction} this operation preserves the USO condition and all $j$-phases.


\begin{proof}[Proof of \Cref{thm:AllHEdgesAreInPhase=>Hypervertex}]
We prove this theorem by contradiction.
Assume $O$ is a USO with a dimension $i \in \dimOfFace$ such that $E(f)_i$ is an $i$-phase of $O$, 
but \face is \emph{not} a hypervertex.
Thus, there exists a dimension $j \in [n] \setminus \dimOfFace$ such that for some pair of vertices $\node, \nodeB \in \face$, the orientation of the connecting edges $\{\node, \node \neig j \}$ and  $\{\nodeB, \nodeB \neig j \}$ differs.

First, we switch our focus to the $n'$-dimensional face $f'$ that contains $f$ and for which $dim(f')=dim(f)\cup\{j\}$. Without loss of generality, we assume $f$ is the lower $j$-facet of $f'$.
Second, we adjust the orientation of the $i$-edges. 
We let all $i$-edges in $E(\face)$ point upwards and all $i$-edges in $E(f') \setminus E(\face)$ point downwards. Since $E(f)_i$ is a phase of $O$, the resulting orientation $O'$ of $f'$ is a USO.

For all $\{v, w\}  \in E(\face)_i$ it holds that $O(v)_j = O(w)_j$, since otherwise the edge $\{v,w\}$ would be in direct phase with the edge $\{v\neig j,w\neig j\}$. Thus, the orientation of the $j$-edges splits $E(f)_i$ into two parts: 
\begin{align*}
E(\face)_i^+ &:= \{ \{v, w\}  \in E(\face)_i \;|\; O'(v)_j = O'(w)_j=0 \} \text{ and} \\
E(\face)_i^- &:=  \{ \{v, w\}  \in E(\face)_i \;|\; O'(v)_j = O'(w)_j = 1 \}.    
\end{align*}

As $E(\face)_i$ is a phase, there must be some edge $\edge^+\in E(\face)_i^+ $ which is in direct phase to some edge $\edge^-\in E(\face)_i^- $. We now perform a partial swap on $\orientation'$ in dimension $j$, yielding a USO $\orientation''$. 
The endpoints of $\edge^+$ are not impacted by this operation, but the endpoints of $\edge^-$ are moved to the opposite $j$-facet, now forming a new edge ${\edge'}^{-}$. The two edges $e^+$ and ${\edge'}^{-}$ must be in direct phase in $O''$, since in dimension $j$ all four of their endpoints are incident to an incoming edge, and in all the other dimensions the outmaps of the endpoints of ${\edge'}^{-}$ in $O''$ are the same as the outmaps of the endpoints of ${\edge}^-$ in $O'$.

\begin{figure}
    \centering
    \ifdefined\arxiv
    \includegraphics{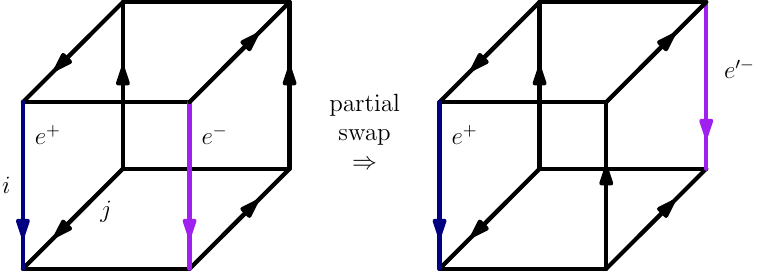}
    \else
    \includegraphics[scale=0.6]{img/proof_partialswap.pdf}
    \fi
    \caption{Sketch of $O'$ and $O''$ from the proof of \Cref{thm:AllHEdgesAreInPhase=>Hypervertex}. The edges $e^+\in E(f)_i^+$ and $e^-\in E(f)_i^-$ are pulled apart by the partial swap. If they were in direct phase before they must still be in direct phase, however the phase connecting $e^+$ and $e'^-$ after the partial swap cannot be connected.}
    \label{fig:counterExample}
\end{figure}

For every pair of $i$-edges neighboring in dimension $j$, exactly one is upwards and one is downwards oriented. Let $P$ be the phase containing $\edge^+$ and ${\edge'}^-$. Since $P$ contains only upwards oriented $i$-edges, and since $P$ contains at least one edge of each $j$-facet, the subgraph of $\neighborhoodGraph_i$ induced by $P$ cannot be connected. See \Cref{fig:counterExample} for a sketch of $O'$ and $O''$.
This is a contradiction to \Cref{thm:connectedness}, which proves the lemma.
\end{proof}


\Cref{thm:AllHEdgesAreInPhase=>Hypervertex} can be adjusted and also holds for phases which are not full faces, only applying to dimensions that leave the minimal face which includes the phase. 

\begin{lemma}
\label{lem:partialHypervertex}
Let $O$ be a USO on the cube \Cube and $P$ an $i$-phase of $O$. Let \face be the minimal face with $P \subseteq E(\face)_i$.
Then, for all dimensions $j \in [n] \setminus \dimOfFace$ and for all $v, w \in P$, we have $O(v)_j = O(w)_j$.   
\end{lemma}

\begin{proof}
The proof is exactly the same as for \Cref{thm:AllHEdgesAreInPhase=>Hypervertex}.
\end{proof}

\else
We prove \Cref{thm:AllHEdgesAreInPhase=>Hypervertex} in the appendix. Note that the proof uses the connectedness of phases (\Cref{thm:connectedness}) as well as the \emph{partial swap} construction from \cite{borzechowski2022construction}. With the same strategy we can also prove a version of \Cref{thm:AllHEdgesAreInPhase=>Hypervertex} for phases that are not full faces:
\begin{lemma}
\label{lem:partialHypervertex}
Let $O$ be a USO on the cube \Cube and $P$ an $i$-phase of $O$. Let \face be the minimal face with $P \subseteq E(\face)_i$.
Then, for all dimensions $j \in [n] \setminus \dimOfFace$ and for all $v, w \in P$, we have $O(v)_j = O(w)_j$.  
\end{lemma}
\fi

It is unclear how --- in the setting of \Cref{lem:partialHypervertex} --- the edges between $P$ and the vertices within $f$ behave. Furthermore it is not known whether some form of \Cref{lem:hypervertexIFFphases} can also be extended to these non-facial subgraph structures.

\subsection{Phases and Matchings} \label{sec:SchurrsProposition}
In \cite{schurr2004phd}, Schurr stated the following proposition, generalizing \Cref{obs:schurrpropositionSingleDim} from sets of edges in the same dimension to \emph{all} matchings:
\begin{proposition}[{\cite[Proposition 4.9]{schurr2004phd}}]\label{prop:schurr}
    Let $O$ be a USO and $H\subseteq E$ be a matching. Then, $O\flip H$ is a USO if and only if $H$ is a union of phases of $O$.
\end{proposition}
Sadly, Schurr's proof of the ``only if'' direction of this proposition is wrong~\cite{Dani}. We reprove the proposition in this section.
\ifdefined\arxiv
Let us first restate the (correct) proof of the ``if'' direction:
\else
Let us first restate the ``if'' direction, which was already proven correctly by Schurr:
\fi
\begin{lemma}[\Cref{prop:schurr}, ``if'']\label{lem:schurrforward}
Let $O$ be a USO and $H\subseteq E$ be a matching that is a union of phases. Then, $O\flip H$ is a USO.
\end{lemma}

\ifdefined\arxiv
\begin{proof}[Proof (from \cite{schurr2004phd})]
    We verify for every pair of vertices $u,v \in V(\Cube)$ that they fulfill the \SWC in $O\flip H$. Both $u$ and $v$ are each incident to at most one edge in $H$.
    
    If neither are incident to an edge in $H$, they trivially fulfill the \SWC in $O\flip H$ since they also did in $O$.
    
    If the one or two edge(s) of $H$ incident to $u$ and $v$ are all $i$-edges for some $i$, then the outmaps of $u$ and $v$ are the same in $O\flip H$ as in $O\flip (H\cap E_i)$, which is a USO by \Cref{obs:schurrpropositionSingleDim} since $H\cap E_i$ is a union of $i$-phases.

    We thus only have to consider the case where $u$ is incident to an $i$-edge of $H$ and $v$ is incident to a $j$-edge of $H$, for $i\neq j$
    and both $i$ and $j$ are within the face that $u$ and $v$ span, i.e.,
    $u_i \neq v_i$ and $u_j \neq v_j$.
    Let $w$ be the vertex in the face spanned by $u$ and $v$ with
    $O(w)\wedge(u\xor v)=(O(u)\neig i)\wedge (u\xor v)$; 
    This means 
    the $i$-edges incident to $u$ and $w$ are in direct phase.
    Such a vertex $w$ is guaranteed to exist since $O$ is a USO and thus a bijection on each face. 
    Now, assume that $u$ and $v$ violate the \SWC in $O\flip H$, i.e., $(u\xor v)\wedge ((O\flip H)(u) \xor (O\flip H)(v))=0^n$. Thus, we must have that in $O$, $(u\xor v)\wedge (O(u)\xor O(v))=I_i\xor I_j$. We can therefore see that the $j$-edges incident to $v$ and $w$ must be in direct phase too. Since $H$ is a union of phases, it must contain both the $i$-edge and the $j$-edge incident to $w$. This contradicts the assumption that $H$ is a matching. Thus, the lemma follows.
\end{proof}

On the other hand, the ``only if'' part can be phrased as follows:
\else
The ``only if'' direction can be phrased as follows:
\fi
\begin{lemma}[\Cref{prop:schurr}, ``only if'']\label{lem:schurrbackward}
Let $O$ be a USO and $H\subseteq E$ a matching. Then, if $O\flip H$ is a USO, $H$ is a union of phases of $O$.
\end{lemma}
Schurr proves this direction by contraposition: Assume $H$ is not a union of phases. Then, there must be a phase $P$ and two edges $e,e'\in P$ such that $e\in H$ and $e'\not\in H$. While $e$ and $e'$ are not necessarily directly in phase, there must be a sequence of direct-in-phaseness relations starting at $e$ and ending in $e'$. At some point in this sequence, there must be two edges $e_i\in H$ and $e_{i+1}\not\in H$ that are in direct phase. Schurr then argues that since we flip only $e_i$ but not $e_{i+1}$, the vertices $v\in e_i$ and $w\in e_{i+1}$ certifying that these edges are in direct phase must violate the \SWC after the flip. Thus, $O\flip H$ would not be a USO. However, it is possible that  $w$ is incident to another edge of a dimension in $v\xor w$ that is contained in $H$, which would make $v$ and $w$ no longer violate the \SWC.

This issue seems very difficult to fix, since the core of the argument (the outmap of $w$ not being changed) is simply wrong. We thus opt to reprove \Cref{lem:schurrbackward} in a completely different way. We first need the following observation:

\begin{observation}
\label{obs:flippingNonAdjecentEdgesDoesNotAffectPhases}
Let \orientation be a USO and 
\Matching a union of phases in \orientation.
Let \Phase be a set of $i$-edges such that $\Matching\cap \Phase=\emptyset$ and $\Matching\cup \Phase$ is a matching. 
If \Phase is a phase in \orientation, it is a union of phases in $\orientation \flip \Matching$.
\end{observation}
\begin{proof}
    If \Phase is a phase in \orientation, by \Cref{lem:schurrforward}, both \orientation with \Matching flipped, and \orientation with $\Matching \cup \Phase$ flipped are USOs. Their difference is \Phase, and thus by \Cref{obs:schurrpropositionSingleDim}, \Phase is a union of phases in $\orientation \flip \Matching$.
\end{proof}

\ifdefined\arxiv
We prove \Cref{lem:schurrbackward} with the help of two \emph{minimization} lemmata, \cref{lem:minimization1,lem:minimization2}. Note that a \emph{counterexample to \Cref{lem:schurrbackward}} is a pair $(O,H)$ such that $O$ is a USO, $H$ is a matching that is not a union of phases in $O$, but $O\flip H$ is a USO nonetheless.

\begin{lemma}[Minimization Lemma 1]\label{lem:minimization1}
If there exists a counterexample $(\orientation^*,\Matching^*)$ to \cref{lem:schurrbackward}, then there also exists a counterexample $(\orientation,\Matching)$ with $dim(\orientation^*)=dim(\orientation)$ in which \Matching does \emph{not} contain a whole phase of \orientation.
\end{lemma}

\begin{proof}
Let $U$ be the union of all phases \Phase of $\orientation^*$ that are fully contained in the matching, i.e., $\Phase \subseteq \Matching^*$. 
By \Cref{lem:schurrforward}, $\orientation:=\orientation^*\flip U$ is a USO. We denote by $\Matching$ the set $\Matching^*\setminus U$, which is a matching containing only incomplete sets of phases of $\orientation^*$.

We now argue that $(\orientation,\Matching)$ is a counterexample with the desired property. 
As $\Matching^*$ originally was not a union of phases, $\Matching\not=\emptyset$. Furthermore, $\orientation\flip \Matching$ is equal to $\orientation^*\flip\Matching^*$ and thus by assumption a USO. It remains to be proven that $\Matching$ is not a union of phases of $\orientation$, and in particular contains no phase completely. 

To do so, we first prove that $U$ is a union of phases in $\orientation$. One can see this by successively flipping in $\orientation^*$ the sets $U_i:=U\cap E_i$ which decompose $U$ into the edges of different dimensions. After flipping each set $U_{i}$, by \Cref{obs:flippingNonAdjecentEdgesDoesNotAffectPhases} all the other sets $U_{j}$ remain unions of phases. Furthermore, as flipping a union of $i$-phases does not change the set of $i$-phases, $U_{i}$ also remains a union of phases.

Now, by \Cref{obs:flippingNonAdjecentEdgesDoesNotAffectPhases}, any phase $\Phase \subseteq \Matching$ of $\orientation$ is a union of phases in $\orientation\flip U = \orientation^*$. But then, by definition of $U$, $P$ would have been included in $U$. Thus, we conclude that $\Matching$ does not contain any phase of $\orientation$.
\end{proof}

\begin{lemma}[Minimization Lemma 2]\label{lem:minimization2}
If there exists a counterexample to \Cref{lem:schurrbackward}, then there also exists a counterexample $(\orientation,\Matching)$ such that for each facet \facet of \orientation, $\Matching \cap \facet$ is a union of phases in \facet, and such that \Matching contains no phase of \orientation.
\end{lemma}
\begin{proof}
    Let $(\orientation,\Matching)$ be a smallest-dimensional counterexample to \Cref{lem:schurrbackward} among all counterexamples $(\orientation^*,\Matching^*)$ where $\Matching^*$ contains no phase of $\orientation^*$. By \Cref{lem:minimization1}, at least one such counterexample must exist, thus $(\orientation,\Matching)$ is well-defined. Let $n$ be the dimension of $\orientation$.
    
    We now prove that $H\cap F$ is a union of phases in $F$ for all facets $F$: If this would not be the case for some $F$, then constraining \orientation and \Matching to \facet would yield a counterexample $(\orientation_\facet,\Matching_\facet)$ of \Cref{lem:schurrbackward} of dimension $n-1$. By applying \Cref{lem:minimization1}, this counterexample can also be turned into a $(n-1)$-dimensional counterexample $(\orientation_\facet',\Matching_\facet')$ such that $\Matching_\facet'$ contains no phase of $\orientation_\facet'$. This is a contradiction to the definition of $(\orientation,\Matching)$ as the smallest-dimensional counterexample with this property. We conclude that $(\orientation,\Matching)$ is a counterexample with the desired properties.
\end{proof}

We are finally ready to prove \Cref{lem:schurrbackward}, and thus also \Cref{prop:schurr}.
\else
We prove \Cref{lem:schurrbackward} with the help of the following \emph{minimization} lemma, the proof of which can be found in the appendix. Note that a \emph{counterexample to \Cref{lem:schurrbackward}} is a pair $(O,H)$ such that $O$ is a USO, $H$ is a matching that is not a union of phases in $O$, but $O\flip H$ is a USO nonetheless.

\begin{lemma}[Minimization Lemma]\label{lem:minimization2}
If there exists a counterexample to \Cref{lem:schurrbackward}, then there also exists a counterexample $(\orientation,\Matching)$ such that for each facet \facet of \orientation, $\Matching \cap \facet$ is a union of phases in \facet, and such that \Matching contains no phase of \orientation.
\end{lemma}

With this lemma we are ready to prove \Cref{lem:schurrbackward}, and thus also \Cref{prop:schurr}.
\fi

\begin{proof}[Proof of \Cref{lem:schurrbackward}]
    By \Cref{lem:minimization2} it suffices to show that there exists no counterexample $(\orientation,\Matching)$ to this lemma such that \Matching contains no phase of \orientation and in each facet \facet of \orientation, $\Matching\cap\facet$ is a union of phases.

    Assume that such a counterexample $(\orientation,\Matching)$ exists. For each dimension $i \in [\dimension]$ for which $\Matching_i:=\Matching\cap E_i$ is non-empty we consider the orientation $\orientation_i := \orientation \flip \Matching_i$. By assumption, $\Matching_i$ is not a union of phases of \orientation, and thus $\orientation_i$ is not a USO. Furthermore, as $\Matching_i$ is a union of phases in each facet of \orientation, all facets of $\orientation_i$ are USOs. Thus, $\orientation_i$ is a PUSO. Recall that by \Cref{lem:PUSO}, in a PUSO, every pair of antipodal vertices has the same outmap. For $\orientation_i$ to be a PUSO, and \orientation to be a USO, exactly one vertex of each antipodal pair of vertices must be incident to an edge in $\Matching_i$.

    We know that \Matching contains edges of at least two dimensions, $i$ and $j$. Consider $\Matching_i$ and $\Matching_j$. By the aforementioned argument, both $\orientation_i$ and $\orientation_j$ are PUSO. As \Matching is a matching, there is no vertex incident to both an edge of $\Matching_i$ and $\Matching_j$. Therefore, for each pair of antipodal vertices $\nodeA, \nodeB$, one is incident to an edge of $\Matching_i$, and one to an edge of $\Matching_j$. 
    Since \nodeA and \nodeB must have the same outmaps in both $\orientation_i$ and $\orientation_j$, we get the following two conditions:
    \begin{enumerate}
        \item $(\nodeA \xor \nodeB ) \wedge (\orientation(\nodeA) \xor \orientation(\nodeB)) = I_i,$ and 
        \item $(\nodeA \xor \nodeB ) \wedge (\orientation(\nodeA) \xor \orientation(\nodeB)) = I_j.$
    \end{enumerate}
    Clearly, this implies $I_i=I_j$ and we have thus obtained a contradiction, and no counterexample can exist to \Cref{lem:schurrbackward}.
\end{proof}

Now that we have recovered \Cref{prop:schurr}, we can also strengthen 
\Cref{obs:flippingNonAdjecentEdgesDoesNotAffectPhases}:

\begin{theorem}
\label{lem:nonAdjacentEdgesAreNotAffectedByPhaseflips}
Let \orientation be a USO and
\Matching a union of phases in \orientation.
Let \Phase be a set of $i$-edges such that $\Matching\cap \Phase=\emptyset$ and $\Matching\cup \Phase$ a matching. 
Then, \Phase is a phase in \orientation if and only if it is a phase in $\orientation' := \orientation \flip \Matching$.
\end{theorem}
\begin{proof}
    We first prove that $P$ is a \emph{union of} phases in \orientation if and only if it is a \emph{union of} phases in $\orientation'$. By \Cref{prop:schurr}, \Matching is a union of phases in both \orientation and $\orientation'$. Thus, this follows from \Cref{obs:flippingNonAdjecentEdgesDoesNotAffectPhases}.

    Now, assume for a contradiction that $P$ is a single phase in \orientation but a union of multiple phases in $\orientation'$ (the other case can be proven symmetrically): Then, let $P'\subset P$ be a phase of $\orientation'$. By the statement proven above, $P'$ is a union of phases in \orientation. However, $P'\subset P$ and $P$ is a single phase. This yields a contradiction, thus the lemma follows.
\end{proof}

We thus conclude that every phase remains a phase when flipping some matching that is not adjacent to any edge of the phase.

\subsection{The \texorpdfstring{$n$}{n}-Schurr Cube}\label{sec:dSchurrCube}
When trying to find an efficient algorithm for computing phases, one might
ask the following: Is there some small integer $k(n)$, such that for every $n$-dimensional USO, the transitive closure of the direct-in-phaseness relation stays the same when only considering the relation between pairs of $i$-edges which have a distance of at most $k(n)$ to each other in $N_i$?
In other words, does it suffice to compute the direct-in-phaseness relationships only for edges that are close to each other, instead of for all pairs of $i$-edges?

In this section we will show that for every $n$ there exists an $n$-dimensional USO in which a direct-in-phase relationship between some antipodal $i$-edges is necessary to define some $i$-phase, i.e., we show that no such $k(n)<n-1$ exists.
Such a USO was found by Schurr~\cite{schurr2004phd} for $n=3$ and is shown in \Cref{fig:3DSchurr}: All four vertical edges are in phase, but if only direct-in-phaseness between non-antipodal edges is considered, there would be no connection between the front and the back facet.

\begin{figure}[htb]
\centering
    \ifdefined\arxiv
    \includegraphics{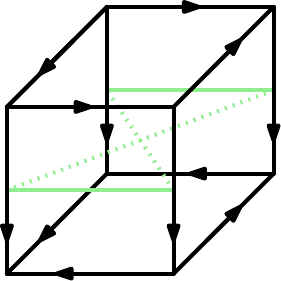}
    \else
    \includegraphics[scale=0.6]{img/3schurr.pdf}
    \fi
\caption{The $3$-Schurr Cube. The vertical edges are all in phase. The direct-in-phase relationships between these edges are marked in green. Note that when disregarding the direct-in-phase relationships of antipodal edges (dotted), this phase splits in two parts.}
\label{fig:3DSchurr}
\end{figure}

We generalize the properties of this 3-Schurr cube to the \dimension-Schurr cube, i.e., we show that there exists an $n$-dimensional USO such that
\begin{itemize}
    \item all $n$-edges are in phase,
    \item all $n$-edges are in direct phase with their antipodal $n$-edge, and
    \item if ignoring the $(n-1)$-direct-in-phaseness, the phase splits apart.
\end{itemize}

We can obtain a cube fulfilling this in a recursive fashion:
\begin{definition}[The $n$-Schurr cube]\label{def:nschurr}
    Let $S^1$ be the $1$-dimensional USO consisting of a single edge oriented towards $0$. For $n\geq 2$, let $S^n$ be the cube obtained by placing $S^{n-1}$ in the lower $n$-facet, and $S^{n-1}\flip E_{n-1}$ in the upper $n$-facet, with all $n$-edges oriented towards the lower $n$-facet.
\end{definition}
Alternatively, we can define the same cube without recursion, simplifying its analysis:
\[\forall v\in \{0,1\}^n:\; S^n(v)_i=\begin{cases}
    v_i\xor v_{i+1} & \text{ for } i<n,\\
    v_n & \text{ for } i=n.
\end{cases}\]

An example of this cube for $n=4$ can be seen in \Cref{fig:4DSchurr}.
This cube fulfills the properties outlined above:

\begin{lemma}
    \label{lem:dschurr}
    In $S^n$ as defined in \Cref{def:nschurr}, all $n$-edges are in direct phase with the antipodal edge (certified by both pairs of antipodal endpoints), and all $n$-edges are in phase. No $n$-edge  is in direct phase with any non-antipodal $n$-edge located in the opposite $1$-facet.
\end{lemma}
\ifdefined\arxiv
\begin{proof}
    We first see that every pair $v,w$ of antipodal vertices certifies their incident $n$-edges to be in direct phase, since for all $i<n$, $v_i\xor v_{i+1}=w_i\xor w_{i+1}$ and thus $S^n(v)_i=S^n(w)_i$.

    Next, we show that no $n$-edge is in direct phase with a non-antipodal $n$-edge in the opposite $1$-facet. Let $v$ be any vertex and $w$ a vertex such that $w_{1}\neq v_{1}$, $w_n\neq v_n$, but $v_i=w_i$ for some $i$. Let $i'$ be the minimum among all $i$ with $v_i=w_i$. Note that $i'>1$. Then, we have $v_{i'-1}\neq w_{i'-1}$, but we also have $S^n(v)_{i'-1}=v_{i'-1}\xor v_{i'} \neq w_{i'-1}\xor w_{i'}=S^n(w)_{i'-1}$, and thus the $n$-edges incident to $v$ and $w$ are not in direct phase.

    By a similar argument one can see that two vertices $v$ and $w$ certify their incident $n$-edges to be in direct phase if and only if there exists some integer $1<k\leq n$ such that $v_i=w_i$ for all $i<k$, and $v_i\neq w_i$ for all $i\geq k$. From this, it is easy to see that all $n$-edges are in phase: The $n$-edges in the upper $1$-facet are each in direct phase with some edge in the lower $1$-facet. The lower $1$-facet is structured in the same way as the cube $S^{n-1}$, thus we can inductively see that all $n$-edges in this facet are in phase. Therefore, all $n$-edges of $S^n$ are in phase.
\end{proof}
\else
The proof of this lemma can be found in the appendix.
\fi

\begin{figure}[htb]
\centering
\begin{subfigure}{0.69\textwidth}
\ifdefined\arxiv
\includegraphics[keepaspectratio,width=.9\textwidth]{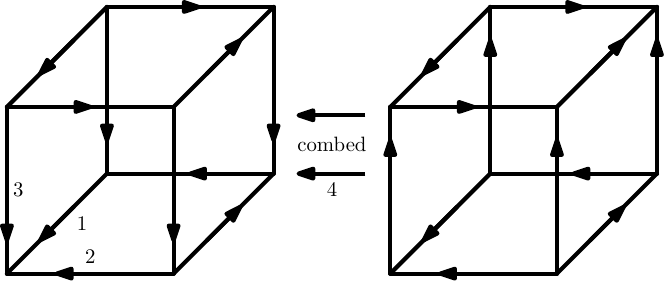}
\else
\includegraphics[keepaspectratio,width=.7\textwidth]{img/4schurr.pdf}
\fi
\end{subfigure}
\begin{subfigure}{0.3\textwidth}
\ifdefined\arxiv
\includegraphics[keepaspectratio,width=.9\textwidth]{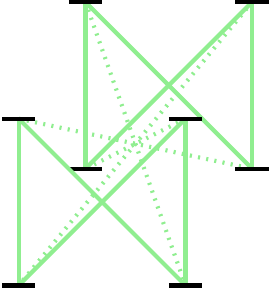}
\else
\includegraphics[keepaspectratio,width=.7\textwidth]{img/4schurrPhases.pdf}
\fi
\end{subfigure}
\caption{The $4$-Schurr Cube. The combed edges between the two pictured $4$-facets are in one phase. The direct-in-phase relationships between these edges is shown on the right. Note that when disregarding the direct-in-phase relationships of antipodal edges (dotted), the phase splits in two parts.}
\label{fig:4DSchurr}
\end{figure}

\section{An Improved Algorithm to Compute Phases}\label{sec:ComputationOfPhases}

The definitions of direct-in-phaseness and phases (\Cref{def:DIP,def:phases}) naturally imply a simple algorithm to compute all phases of a USO: Compare every pair of vertices and record the edges that are in direct phase, then run a connected components algorithm on the graph induced by these direct-in-phase relationships. This takes $O(4^n)$ time for an $n$-dimensional cube. As we will see, we can do better.

Based on \Cref{obs:schurrpropositionSingleDim} we get a natural connection between USO recognition and computation of phases; if $O$ is a USO and $H\subseteq E_i$ a set of $i$-edges, then $H$ is a union of phases if and only if $O\flip H$ is a USO. However, using USO recognition as a black-box algorithm would be highly inefficient for computing phases (as opposed to testing whether some set is a phase), since we would have to check whether $O\flip H$ is a USO for many different candidate sets $H$. We will see that a single run of an USO recognition algorithm suffices to be able to compute all phases.

To achieve this, we can profit from the similarity of the \SWC (\Cref{lem:szabowelzl}) and the condition for being in direct phase (\Cref{def:DIP}). Let $\mathcal{A}$ be an algorithm for USO recognition that tests the \SWC for some subset $T$ of all vertex pairs $\binom{V(Q_n)}{2}$, and outputs that the given orientation is a USO if and only if all pairs in $T$ fulfill the \SWC. We will show that then there exists an algorithm $\mathcal{B}$ for computing all phases of a given USO that also only compares the vertex pairs $T$.
Our phase computation algorithm $\mathcal{B}$ is based on the following symmetric relation:
\begin{definition}
Let $T\subseteq \binom{V(Q_n)}{2}$. Two $i$-edges $e,e'$ are \emph{in direct $T$-phase}, if
\begin{itemize}
    \item $e$ and $e'$ are in direct phase and
    \item there exist $v \in e$, $v' \in e'$ such that $v,v'$ certify $e,e'$ to be in direct phase, and $\{v,v'\}\in T$.
\end{itemize}
\end{definition}

\begin{lemma}\label{lem:blerbs}
Let $T\subseteq \binom{V(Q_n)}{2}$ be the set of vertex pairs of a USO recognition algorithm~$\mathcal{A}$. In every USO $O$, the transitive closure of direct-in-$T$-phaseness is equal to the transitive closure of direct-in-phaseness.
\end{lemma}
\begin{proof}
Clearly, by definition, the equivalence classes of direct-in-$T$-phaseness (called \emph{$T$-phases}) are a refinement of phases. We now show that every $T$-phase is also a phase. Assume there is a $T$-phase $B$ (in dimension $i$) that is a strict subset of a phase $P$. Then, by \Cref{obs:schurrpropositionSingleDim}, $O\flip B$ is not a USO. Algorithm $\mathcal{A}$ must be able to detect this. Thus, there is a vertex pair $\{v,v'\}\in T$ that violates the \SWC in $O\flip B$. Clearly, exactly one of the $i$-edges incident to $v,v'$ is contained in $B$. However, these edges are in direct phase in $O$, and thus they are also in direct $T$-phase. We conclude that both edges should be in the same $T$-phase, and thus have a contradiction.
\end{proof}

With this lemma, we can turn the USO recognition algorithm $\mathcal{A}$ into a phase computation algorithm $\mathcal{B}$:
For each pair of vertices in $T$, find the edges they certify to be in direct $T$-phase. Then, calculate the connected components.

The best known USO recognition algorithm --- the one based on the work of Bosshard and Gärtner on PUSOs~\cite{bosshard2017pseudo} --- uses a set $T$ of size $|T|=3^n$: this follows from the fact that \Cref{lem:PUSO} implies that for each face, only the minimum and maximum vertex need to be compared.
Thus, we can also calculate the phases in $O(3^n)$ time, and any further advances in USO recognition will also imply further improvements in phase computation.

\ifdefined\arxiv
The fact that \Cref{lem:blerbs} holds for \emph{all} valid USO recognition algorithms may also be used to derive some structural results on phases. For example, \Cref{lem:blerbs} applied to the $2^n$ versions of the PUSO-based USO recognition algorithm (each version specified by the vertex $v$ which is interpreted as the minimum vertex of the cube) implies the following lemma:

\begin{lemma}
\label{lem:dist-k-dip}
Let $O$ be an $n$-dimensional USO and let $P$ be an $i$-phase. If $P$ would split apart if only direct-in-phase relationships of edges of distance $<n-1$ are considered, then all antipodal vertices must certify their incident $i$-edges to be in direct phase, and $P=E_i$.
\end{lemma}
\begin{proof}
    Every version of the PUSO-based algorithm only checks one pair of antipodal vertices. By \Cref{lem:blerbs}, each of these pairs must certify their incident $i$-edges to be in direct phase (since ignoring direct-in-phase relationships of antipodal vertices must split apart $P$ by assumption). Furthermore, since this in-phaseness must be relevant for $P$, both of these edges must be part of $P$, and thus all $i$-edges must be in $P$.
\end{proof}
\fi

While we were able to slightly improve the runtime of computing phases, we show in the next section that this likely cannot be improved much further, since checking whether two given edges are in phase is \PSPACE-complete.

\section{\PSPACE-Completeness}\label{sec:completeness}

Checking whether two edges are in \emph{direct} phase in a USO is trivial, it can be achieved with just four evaluations of the outmap function and $O(n)$ additional time. Surprisingly, in this section we prove that testing whether two edges are in phase (not necessarily directly) is \PSPACE-complete. We first have to make the computational model more clear: Since a USO is a graph of exponential size (in the dimension $n$), the usual way of specifying a USO is by a \emph{succinct} representation, i.e., a Boolean circuit computing the outmap function with $n$ inputs and $n$ outputs and overall size polynomial in $n$. This reflects the practical situation very well, since in all current applications of USO sink-finding, the outmap function can be evaluated in time polynomial in $n$.
\begin{definition}
    The decision problem 2IP is to decide the following question:\\
Given a USO \orientation by a Boolean circuit of size in $O(poly(n))$ and two edges $\edgeA, \edgeB$, are \edgeA and \edgeB in phase?

\end{definition}

We first show that 2IP can be solved in polynomial space.
\begin{lemma}
2IP is in \PSPACE.
\end{lemma}
\begin{proof}
We show that 2IP can be solved in polynomial space on a nondeterministic Turing machine, i.e., 2IP $\in\NPSPACE$. By Savitch's theorem \cite{Savitch}, \NPSPACE=\PSPACE.
We solve 2IP by starting at the edge $e$ and guessing a sequence of edges that are each in direct phase with the previous edge. If in this way we can reach $e'$, $e$ and $e'$ must be in phase. Such an algorithm only needs $O(n)$ bits to store the current and the next guessed edge of the sequence.
\end{proof}

Next, we show that 2IP is \PSPACE-hard. Since our reduction will only generate acyclic USOs, the problem remains \PSPACE-hard even under the promise that the input function specifies an acyclic USO. This restriction makes the theorem much more powerful, since testing whether this promise holds is itself \PSPACE-complete~\cite{gaertner2015recognizing}.
For our proof we reduce from the following (standard) \PSPACE-complete problem:
\begin{definition}
    The \emph{Quantified Boolean Formula} (QBF) is to decide the following: Given a formula $\Phi$ in conjunctive normal form on the variables $x_1,\ldots,x_n$, as well as a set of quantifiers $q_1,\ldots,q_n\in\{\exists,\forall\}$, decide whether the sentence $q_1x_1,\ldots,q_nx_n:\Phi(x_1,\ldots,x_n)$ is true.
\end{definition}

\begin{fact}
    QBF is \PSPACE-complete.
\end{fact}

\begin{theorem}
2IP is \PSPACE-hard, even when the input is guaranteed to be an acyclic USO and $e$ and $e'$ are antipodal.
\end{theorem}

\begin{proof}
We reduce QBF to 2IP. To prove \PSPACE-hardness, this reduction must be polynomial-time and many-one. We translate a sentence $q_1x_1,\ldots,q_nx_n:\Phi(x_1,\ldots,x_n)$ into an acyclic USO $O_0[]$, built recursively from the USOs $O_{1}[0]$ and $O_{1}[1]$ which correspond to the sentences $q_2x_2,\ldots,q_{n}x_{n}:\Phi(0,x_2,\ldots,x_{n})$ and $q_2x_2,\ldots,q_{n}x_{n}:\Phi(1,x_2,\ldots,x_{n})$, respectively. In general, a USO $O_i[b_{1},\ldots,b_{i}]$ for $b_j\in\{0,1\}$ corresponds to the sentence $q_{i+1}x_{i+1},\ldots,q_nx_n:\Phi(b_{1},\ldots,b_i,x_{i+1},\ldots,x_n)$.

We show inductively that all of our orientations $O_i[b_{1},\ldots,b_i]$ fulfill the following invariants:
\begin{itemize}[noitemsep]
\item $O_i[b_{1},\ldots,b_i]$ is a USO.
\item $O_i[b_{1},\ldots,b_i]$ is acyclic.
\item $O_i[b_{1},\ldots,b_i]$ is combed downwards in dimension $1$.
\item The minimum vertex of $O_i[b_{1},\ldots,b_i]$ is its sink, the maximum vertex is its source.
\item In $O_i[b_{1},\ldots,b_i]$, the $1$-edges incident to the minimum and maximum vertices are in phase if and only if the sentence $q_{i+1}x_{i+1},\ldots,q_nx_n:\Phi(b_{1},\ldots,b_i,x_{i+1},\ldots,x_n)$ is true.
\end{itemize}
If we can show these properties, the only step left for the proof of the reduction is to show that a circuit computing $O_0[]$ can be computed in polynomial time.

We first begin by discussing the anchor of our recursive construction: The orientations $O_n[b_1,\ldots,b_n]$, which correspond to the (unquantified) sentences $\Phi(b_1,\ldots,b_n)$. The truth value of such an unquantified sentence can be efficiently tested (one simply needs to evaluate~$\Phi$ once), and we can thus set these orientations to be one of two fixed orientations: the \emph{true-} or the \emph{false-gadget}.

\proofsubparagraph{Base Gadgets}
\begin{figure}[htb]
\begin{subfigure}{0.49\textwidth}
    \ifdefined\arxiv
    \includegraphics{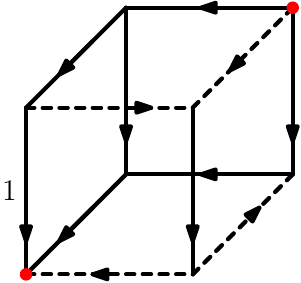}
    \else
    \includegraphics[scale=0.6]{img/pspace/true_base.pdf}
    \fi    \caption{The gadget encoding true.}
    \label{subfig:base_true}
\end{subfigure}
\begin{subfigure}{0.49\textwidth}
    \ifdefined\arxiv
    \includegraphics{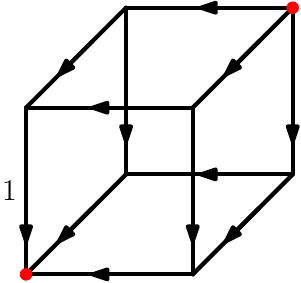}
    \else
    \includegraphics[scale=0.6]{img/pspace/false_base.pdf}
    \fi
    \caption{The gadget encoding false.}
    \label{subfig:base_false}
\end{subfigure}
\caption{The two base gadgets.}\label{fig:base_gadgets}
\end{figure}

The two base gadgets, the true- and the false-gadget, are the $3$-dimensional USOs shown in \Cref{fig:base_gadgets}. As can be seen, they are both acyclic USOs with sink and source at the minimum and maximum vertex, respectively, and combed downwards in dimension $1$. In the true-gadget, the minimum and maximum vertex of each $1$-facet are connected by a path of two edges (dashed) whose orientations are different in the upper and lower $1$-facets. Thus, along this path, the incident $1$-edges are always in direct phase. We can thus see that the $1$-edges incident to the minimum and maximum vertices must be in phase, as required. In contrast, in the false-gadget every $1$-edge is flippable (since the gadget is just a uniform USO), and thus the $1$-edges incident to the minimum and maximum vertices are not in phase.
We thus conclude that the base cases of our induction hold.

\proofsubparagraph{The \texorpdfstring{$\forall$}{``for all''} Quantifier}
We now show how we build a USO $\mathcal{O}:=O_i[b_{1},\ldots,b_i]$, if $q_{i+1}=\forall$. We first note that $q_{i+1}x_{i+1},\ldots,q_nx_n:\Phi(b_{1},\ldots,b_i,x_{i+1},\ldots,x_n)$ is true if and only if \emph{both} of the sentences $q_{i+2}x_{i+2},\ldots,q_nx_n:\Phi(b_{1},\ldots,b_i,B,x_{i+2},\ldots,x_n)$ for $B\in\{0,1\}$, i.e., the two sentences corresponding to $\mathcal{F}:=O_{i+1}[b_{1},\ldots,b_i,0]$ and $\mathcal{T}:=O_{i+1}[b_{1},\ldots,b_i,1]$, are true.

\begin{figure}
    \centering
    \ifdefined\arxiv
    \includegraphics{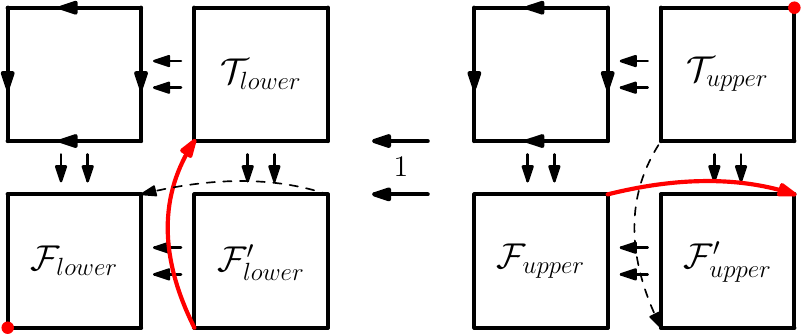}
    \else
    \includegraphics[scale=0.5]{img/pspace/forall_recursive.pdf}
    \fi
    \caption{The $\forall$ construction. The $1$-edges incident to the red vertices are only in phase if the minimum and maximum $1$-edges of both $\mathcal{F}$ and $\mathcal{T}$ are in phase.}
    \label{fig:for_all_gadget}
\end{figure}
We show how to build $\mathcal{O}$ from $\mathcal{T}$ and $\mathcal{F}$ in \Cref{fig:for_all_gadget}. Essentially, $\mathcal{O}$ consists of two copies of $\mathcal{F}$ ($\mathcal{F}$ and $\mathcal{F}'$), one copy of $\mathcal{T}$, and one uniform USO, all connected in a combed way, but then two specially marked red edges are flipped. Note that by the inductive hypothesis, we can assume both $\mathcal{F}$ and $\mathcal{T}$ to fulfill the invariants outlined above.

We first prove that $\mathcal{O}$ is a USO: Since all four ``ingredients'' are USOs, before flipping the red edges the orientation is clearly a USO (it can be seen as a product construction as described in~\cite{schurr2004quadraticbound}). We thus only have to show that the two red edges are flippable. The red edge in the top $1$-facet goes between the maximum vertices of $\mathcal{F}_{upper}$ and $\mathcal{F}'_{upper}$, which by inductive hypothesis are both sources of their respective subcubes. Thus, the two endpoints have the same outmap, and this red edge is flippable. The same argument works for the red edge in the bottom $1$-facet, which goes between two sinks. Thus, the orientation is a USO.

Next, we prove that the construction preserves acyclicity: We can view both $1$-facets independently, since the $1$-edges are combed and thus cannot be part of a cycle. In a similar way, we can split each $1$-facet further along some combed dimension. In the resulting subcubes, since the uniform USO, $\mathcal{F}$, and $\mathcal{T}$ are acyclic, any cycle must use one of the red edges. However, in these subcubes each red edge either ends at a sink or starts at a source, and can thus not be part of any directed cycle. Thus, $\mathcal{O}$ is acyclic.

Next, we want to point out that $\mathcal{O}$ is combed downwards in dimension $1$, and since the minimum vertex of $\mathcal{F}$ is a sink and the maximum vertex of $\mathcal{T}$ is a source, the sink and source of $\mathcal{O}$ are also located at the minimum and maximum vertex, respectively.

Finally, we need to show that the $1$-phases of $\mathcal{O}$ are correct. In other words, we wish to prove that the $1$-edges incident to the minimum and maximum vertices are in phase if and only if this holds for \emph{both} $\mathcal{F}$ and $\mathcal{T}$.\\
The ``if'' direction is easy to see, since we have a chain of in-phaseness: We can first go through $\mathcal{F}$, then cross over to the right (since the red and dashed edges go in the opposite direction, their incident $1$-edges are in phase), take the same path back through $\mathcal{F}'$, cross upwards along the red and dashed edges, and finally go through $\mathcal{T}$.\\
For the ``only if'' direction, we can assume that the $1$-edges incident to the minimum and maximum vertices are not in phase in either $\mathcal{F}$ or $\mathcal{T}$. Thus, there must be some phase $P$ in $\mathcal{F}$ that includes the $1$-edge incident to its source but not the one incident to its sink, or there exists a phase $P$ in $\mathcal{T}$ including the $1$-edge incident to its sink but not the one incident to its source. This phase $P$ forms a matching even when the two flippable red edges are added. Thus, by \Cref{lem:schurrforward}, we can flip $P$ also in $\mathcal{O}$. However, $P$ contains exactly one of the two $1$-edges incident to the minimum and maximum vertices of $\mathcal{O}$. Thus, these $1$-edges are not in phase.\\
Thus, we conclude that the $1$-edges incident to the minimum and maximum vertices are in phase if and only if this also held for $\mathcal{F}$ and $\mathcal{T}$.

\proofsubparagraph{The \texorpdfstring{$\exists$}{``exists''} Quantifier}
Now we show how we build a USO $\mathcal{O}:=O_i[b_{1},\ldots,b_i]$, if $q_{i+1}=\exists$. We again note that $q_{i+1}x_{i+1},\ldots,q_nx_n:\Phi(b_{1},\ldots,b_i,x_{i+1},\ldots,x_n)$ is true if and only if \emph{at least one} of the sentences $q_{i+2}x_{i+2},\ldots,q_nx_n:\Phi(b_{1},\ldots,b_i,B,x_{i+2},\ldots,x_n)$ for $B\in\{0,1\}$, i.e., the two sentences corresponding to $\mathcal{F}:=O_{i+1}[b_{1},\ldots,b_i,0]$ and $\mathcal{T}:=O_{i+1}[b_{1},\ldots,b_i,1]$, are true.

\begin{figure}[htb]
    \centering
    \ifdefined\arxiv
    \includegraphics{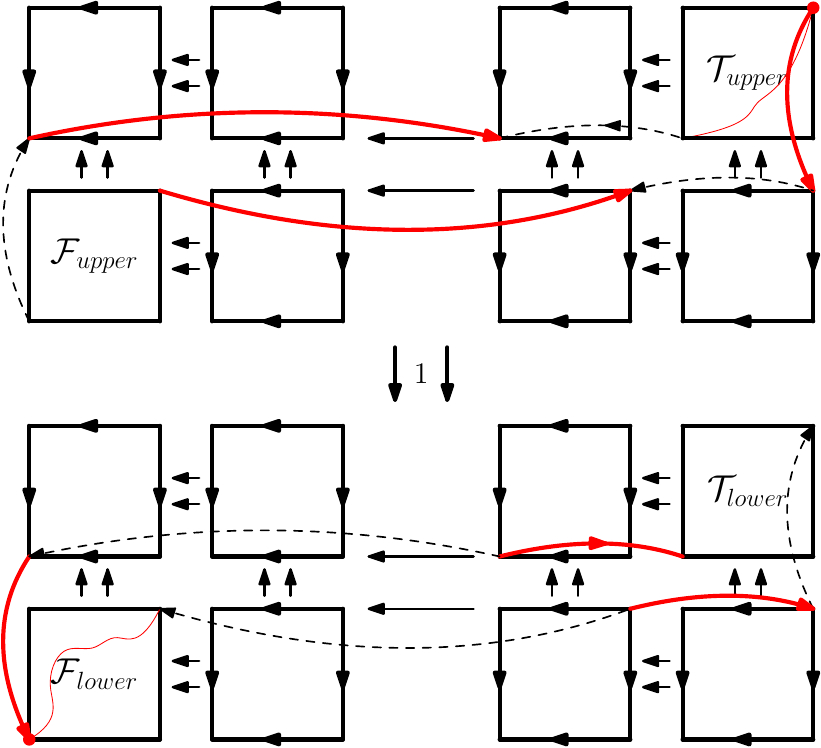}
    \else
    \includegraphics[scale=0.5]{img/pspace/exists_recursive.pdf}
    \fi
    \caption{The $\exists$ construction. The $1$-edges incident to the red vertices are in phase if and only if the minimum and maximum $1$-edges of either $\mathcal{F}$ or $\mathcal{T}$ are in phase.}
    \label{fig:exists_gadget}
\end{figure}

We show how to build $\mathcal{O}$ from $\mathcal{T}$ and $\mathcal{F}$ in \Cref{fig:exists_gadget}. Essentially, $\mathcal{O}$ consists of one copy of $\mathcal{F}$, one copy of $\mathcal{T}$, and six uniform USOs, all connected in a combed way, but then six specially marked red edges are flipped. Note that by the inductive hypothesis, we can again assume both $\mathcal{F}$ and $\mathcal{T}$ to fulfill the conditions outlined above.

We first prove that $\mathcal{O}$ is a USO: Similarly to the $\forall$ construction, we only need to show that all the red edges are flippable. Again, this follows from the location of the sources and sinks of $\mathcal{F}$ and $\mathcal{T}$. Furthermore, the red edges form a matching, so they can also all be flipped together.

Next, we see that the construction preserves acyclicity: We can again decompose $\mathcal{O}$ into subcubes along combed dimensions. In the remaining subcubes, we see that all red edges are incident to sinks or sources, and can thus not be used in directed cycles. It follows that $\mathcal{O}$ is acyclic.

It is again obvious to see that $\mathcal{O}$ is combed downwards in dimension $1$, and the global sink and source of $\mathcal{O}$ are located in the minimum and maximum vertex, respectively.

Finally, we show that the $1$-edges incident to the minimum and maximum vertices of $\mathcal{O}$ are in phase if and only if this holds for \emph{at least one of} $\mathcal{F}$ and $\mathcal{T}$.\\
The ``if'' direction is again simple, since we can use the in-phaseness sequence through $\mathcal{F}$ and the three lower red and dashed edge pairs, or the in-phaseness sequence through $\mathcal{T}$ and the upper three red and dashed edge pairs.\\
For the ``only if'' direction, we see that all $1$-edges outside of $\mathcal{F}$ and $\mathcal{T}$ are flippable, except the four that are adjacent to a red edge. We can check all pairs of remaining $1$-edges for possibly being in direct phase and verify that only the direct-in-phaseness relations induced by the red and dashed edge pairs and the relations inside of $\mathcal{F}$ and $\mathcal{T}$ are present. Thus, the minimum and maximum $1$-edges of $\mathcal{O}$ can only be in phase, if that holds for $\mathcal{F}$ or $\mathcal{T}$.

\proofsubparagraph{Final Arguments}
It only remains to prove that we can build a circuit computing the outmap function $O_0[]$ from the QBF instance in polynomial time. Based on the sequence $q_1,\ldots,q_n$ of quantifiers, we can easily assign the dimensions of $O_0[]$ to the different levels of the recursion (the first three coordinates belong to the base gadgets, and the following coordinates belong to levels in either groups of two or three coordinates, depending on whether $q_i=\exists$ or $q_i=\forall$). We can thus easily locate a given vertex $v$ within all levels of the recursive construction. If at some point the vertex is part of a uniform subcube, it does not need to be located on lower levels. Otherwise, the vertex is part of a base gadget $O_n[b_1,\ldots,b_n]$ on the last level of the recursion. Here, we can evaluate $\Phi(b_1,\ldots,b_n)$ (since CNF formulae can be efficiently evaluated by Boolean circuits), and find the orientation of $v$. 
Thus, $O_0[]$ can be evaluated by a polynomially sized circuit that we can also build in polynomial time, and our reduction is complete.
\end{proof}

\subsection{Implications}
The \PSPACE-hardness of 2IP implies that many closely related problems concerning phases are also hard, for example, computing the set of edges in the phase $P$ in which a given edge lies. More surprisingly, its hardness also implies that the problem of \emph{USO completion} is \PSPACE-hard. In the USO completion problem, one is given a partially oriented hypercube. This partial orientation is again encoded by a succinct circuit, and computes for each vertex $v$ its \emph{partial outmap} as a function $C:V(Q_n)\rightarrow\{0,1,-\}^n$ where $0$ and $1$ denote incoming and outgoing edges as usual, and ``$-$'' denotes that the edge is not oriented. The problem is then to decide whether there exists a USO $O$ that agrees with $C$ on all edges that were oriented in $C$. It is easy to obtain a reduction from 2IP to USO completion: In the dimension of the two input $h$-edges $e,e'$, we make all edges unoriented, except $e$ and $e'$, which are oriented in opposite directions. Clearly, if this partial orientation is completable, $e$ and $e'$ cannot be in phase. If this orientation is not completable, $e$ and $e'$ must be directed in the same way in all possible connections of the two $h$-facets, i.e., they must be in the same $h$-phase.
\section{Conclusion}

Since implementations of the USO Markov chain spend most of their time computing phases, our improvement from $O(4^n)$ to $O(3^n)$ vertex comparisons significantly speeds up the generation of random USOs in practice. It is conceivable that phases could be computed even faster, but $\Omega(n2^n)$ serves as a natural lower bound due to the number of edges of $\Cube$. The main open question in this area remains the mixing rate of the Markov chain, and we hope that some of our structural results may serve as new tools towards attacking this problem. Currently phases seem to be the only somewhat ``local'' rule to generate all USOs, but it may also prove useful to search for other operations which allow for efficient random sampling.

Some of our results %
\ifdefined\arxiv
(\Cref{obs:EveryPairOfVerticesCertifiesDIPOfAtMost2Edges}, \Cref{lem:partialHypervertex}, \Cref{lem:nonAdjacentEdgesAreNotAffectedByPhaseflips})
\else
(\Cref{lem:partialHypervertex}, \Cref{lem:nonAdjacentEdgesAreNotAffectedByPhaseflips})
\fi 
indicate that the phases or direct-in-phaseness relationships of one dimension contain information on the phases in other dimension. It is also not known how the sizes of the $i$-phases affect the number and sizes of phases in other dimensions, apart from the fact that there are at least $2n$ phases in total. Such interactions between phases of different dimensions call for further study, and might help with phase computation in the future; for example it may be possible to efficiently deduce the $i$-phases given all the $j$-phases for $j\neq i$.


\clearpage
\bibliography{literature,USO}

\ifdefined\arxiv
\else

\newpage
\appendix
\section{Omitted Proofs}
\subsection{\Cref{thm:AllHEdgesAreInPhase=>Hypervertex,lem:partialHypervertex}}
\begin{proof}[Proof of \Cref{thm:AllHEdgesAreInPhase=>Hypervertex}]
We prove this theorem by contradiction.
Assume $O$ is a USO with a dimension $i \in \dimOfFace$ such that $E(f)_i$ is an $i$-phase of $O$, 
but \face is \emph{not} a hypervertex.
Thus, there exists a dimension $j \in [n] \setminus \dimOfFace$ such that for some pair of vertices $\node, \nodeB \in \face$, the orientation of the connecting edges $\{\node, \node \neig j \}$ and  $\{\nodeB, \nodeB \neig j \}$ differs.

First, we switch our focus to the $n'$-dimensional face $f'$ that contains $f$ and for which $dim(f')=dim(f)\cup\{j\}$. Without loss of generality, we assume $f$ is the lower $j$-facet of $f'$.
Second, we adjust the orientation of the $i$-edges.
We let all $i$-edges in $E(\face)$ point upwards and all $i$-edges in $E(f') \setminus E(\face)$ point downwards. Since $E(f)_i$ is a phase of $O$, the resulting orientation $O'$ of $f'$ is a USO.

For all $\{v, w\}  \in E(\face)_i$ it holds that $O(v)_j = O(w)_j$, since otherwise the edge $\{v,w\}$ would be in direct phase with the edge $\{v\neig j,w\neig j\}$. Thus, the orientation of the $j$-edges splits $E(f)_i$ into two parts: 
\begin{align*}
E(\face)_i^+ &:= \{ \{v, w\}  \in E(\face)_i \;|\; O'(v)_j = O'(w)_j=0 \} \text{ and} \\
E(\face)_i^- &:=  \{ \{v, w\}  \in E(\face)_i \;|\; O'(v)_j = O'(w)_j = 1 \}.    
\end{align*}

As $E(\face)_i$ is a phase, there must be some edge $\edge^+\in E(\face)_i^+ $ which is in direct phase to some edge $\edge^-\in E(\face)_i^- $. We now perform a partial swap on $\orientation'$ in dimension $j$, yielding a USO $\orientation''$. 
The endpoints of $\edge^+$ are not impacted by this operation, but the endpoints of $\edge^-$ are moved to the opposite $j$-facet, now forming a new edge ${\edge'}^{-}$. The two edges $e^+$ and ${\edge'}^{-}$ must be in direct phase in $O''$, since in dimension $j$ all four of their endpoints are incident to an incoming edge, and in all the other dimensions the outmaps of the endpoints of ${\edge'}^{-}$ in $O''$ are the same as the outmaps of the endpoints of ${\edge}^-$ in $O'$.

For every pair of $i$-edges neighboring in dimension $j$, exactly one is upwards and one is downwards oriented. Let $P$ be the phase containing $\edge^+$ and ${\edge'}^-$. Since $P$ contains only upwards oriented $i$-edges, and since $P$ contains at least one edge of each $j$-facet, the subgraph of $\neighborhoodGraph_i$ induced by $P$ cannot be connected. See \Cref{fig:counterExample} for a sketch of $O'$ and $O''$.
This is a contradiction to \Cref{thm:connectedness}, which proves the lemma.
\end{proof}

\begin{figure}[hb]
    \centering
    \ifdefined\arxiv
    \includegraphics{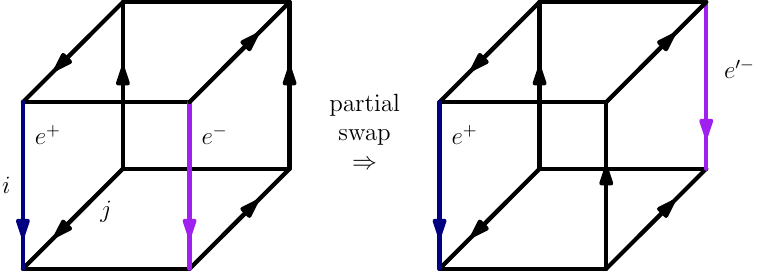}
    \else
    \includegraphics[scale=0.6]{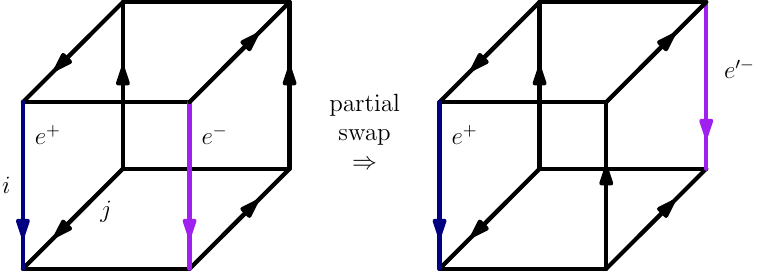}
    \fi
    \caption{Sketch of $O'$ and $O''$ from the proof of \Cref{thm:AllHEdgesAreInPhase=>Hypervertex}. The edges $e^+\in E(f)_i^+$ and $e^-\in E(f)_i^-$ are pulled apart by the partial swap. If they were in direct phase before they must still be in direct phase, however the phase connecting $e^+$ and $e'^-$ after the partial swap cannot be connected.}
    \label{fig:counterExample}
\end{figure}

\begin{proof}[Proof of \Cref{lem:partialHypervertex}]
The proof is exactly the same as for \Cref{thm:AllHEdgesAreInPhase=>Hypervertex}.
\end{proof}

\subsection{\Cref{lem:minimization2}}
We prove this lemma with a weaker minimization lemma:
\begin{lemma}[Weak Minimization Lemma]\label{lem:minimization1}
If there exists a counterexample $(\orientation^*,\Matching^*)$ to \cref{lem:schurrbackward}, then there also exists a counterexample $(\orientation,\Matching)$ with $dim(\orientation^*)=dim(\orientation)$ in which \Matching does \emph{not} contain a whole phase of \orientation.
\end{lemma}
\begin{proof}
Let $U$ be the union of all phases \Phase of $\orientation^*$ that are fully contained in the matching, i.e., $\Phase \subseteq \Matching^*$. 
By \Cref{lem:schurrforward}, $\orientation:=\orientation^*\flip U$ is a USO. We denote by $\Matching$ the set $\Matching^*\setminus U$, which is a matching containing only incomplete sets of phases of $\orientation^*$.

We now argue that $(\orientation,\Matching)$ is a counterexample with the desired property. 
As $\Matching^*$ originally was not a union of phases, $\Matching\not=\emptyset$. Furthermore, $\orientation\flip \Matching$ is equal to $\orientation^*\flip\Matching^*$ and thus by assumption a USO. It remains to be proven that $\Matching$ is not a union of phases of $\orientation$, and in particular contains no phase completely. 

To do so, we first prove that $U$ is a union of phases in $\orientation$. One can see this by successively flipping in $\orientation^*$ the sets $U_i:=U\cap E_i$ which decompose $U$ into the edges of different dimensions. After flipping each set $U_{i}$, by \Cref{obs:flippingNonAdjecentEdgesDoesNotAffectPhases} all the other sets $U_{j}$ remain unions of phases. Furthermore, as flipping a union of $i$-phases does not change the set of $i$-phases, $U_{i}$ also remains a union of phases.

Now, by \Cref{obs:flippingNonAdjecentEdgesDoesNotAffectPhases}, any phase $\Phase \subseteq \Matching$ of $\orientation$ is a union of phases in $\orientation\flip U = \orientation^*$. But then, by definition of $U$, $P$ would have been included in $U$. Thus, we conclude that $\Matching$ does not contain any phase of $\orientation$.
\end{proof}

\begin{proof}[Proof of \Cref{lem:minimization2}]
    Let $(\orientation,\Matching)$ be a smallest-dimensional counterexample to \Cref{lem:schurrbackward} among all counterexamples $(\orientation^*,\Matching^*)$ where $\Matching^*$ contains no phase of $\orientation^*$. By \Cref{lem:minimization1}, at least one such counterexample must exist, thus $(\orientation,\Matching)$ is well-defined. Let $n$ be the dimension of $\orientation$.
    
    We now prove that $H\cap F$ is a union of phases in $F$ for all facets $F$: If this would not be the case for some $F$, then constraining \orientation and \Matching to \facet would yield a counterexample $(\orientation_\facet,\Matching_\facet)$ of \Cref{lem:schurrbackward} of dimension $n-1$. By applying \Cref{lem:minimization1}, this counterexample can also be turned into a $(n-1)$-dimensional counterexample $(\orientation_\facet',\Matching_\facet')$ such that $\Matching_\facet'$ contains no phase of $\orientation_\facet'$. This is a contradiction to the definition of $(\orientation,\Matching)$ as the smallest-dimensional counterexample with this property. We conclude that $(\orientation,\Matching)$ is a counterexample with the desired properties.
\end{proof}

\subsection{\Cref{lem:dschurr}}
\begin{proof}[Proof of \Cref{lem:dschurr}]
    We first see that every pair $v,w$ of antipodal vertices certifies their incident $n$-edges to be in direct phase, since for all $i<n$, $v_i\xor v_{i+1}=w_i\xor w_{i+1}$ and thus $S^n(v)_i=S^n(w)_i$.

    Next, we show that no $n$-edge is in direct phase with a non-antipodal $n$-edge in the opposite $1$-facet. Let $v$ be any vertex and $w$ a vertex such that $w_{1}\neq v_{1}$, $w_n\neq v_n$, but $v_i=w_i$ for some $i$. Let $i'$ be the minimum among all $i$ with $v_i=w_i$. Note that $i'>1$. Then, we have $v_{i'-1}\neq w_{i'-1}$, but we also have $S^n(v)_{i'-1}=v_{i'-1}\xor v_{i'} \neq w_{i'-1}\xor w_{i'}=S^n(w)_{i'-1}$, and thus the $n$-edges incident to $v$ and $w$ are not in direct phase.

    By a similar argument one can see that two vertices $v$ and $w$ certify their incident $n$-edges to be in direct phase if and only if there exists some integer $1<k\leq n$ such that $v_i=w_i$ for all $i<k$, and $v_i\neq w_i$ for all $i\geq k$. From this, it is easy to see that all $n$-edges are in phase: The $n$-edges in the upper $1$-facet are each in direct phase with some edge in the lower $1$-facet. The lower $1$-facet is structured in the same way as the cube $S^{n-1}$, thus we can inductively see that all $n$-edges in this facet are in phase. Therefore, all $n$-edges of $S^n$ are in phase.
\end{proof}
\fi

\end{document}